\documentclass[a4paper,11pt]{amsart}

\usepackage{latexsym}
\usepackage{amssymb,amsmath}
\usepackage{graphics, tikz}
\usepackage{url}
\usepackage[vcentermath]{youngtab}
  
\usepackage{booktabs}  

\newtheorem{theorem}{Theorem}
\newtheorem*{theorem*}{Theorem}

\newtheorem{lemma}[theorem]{Lemma}

\theoremstyle{definition}

\newcommand{\N}{\mathbf{N}}

\newcommand{\R}{\mathbf{R}}

\newcommand{\E}{\mathbf{E}}
\renewcommand{\P}{\mathbf{P}}
\DeclareMathOperator{\Var}{\mathbf{Var}}

\renewcommand{\epsilon}{\varepsilon}

\renewcommand{\theta}{\vartheta}

\newcommand{\sfrac}[2]{{\scriptstyle\frac{#1}{#2}}}
\newcommand{\mfrac}[2]{{\textstyle\frac{#1}{#2}}}

\newcounter{thmlistcnt}
	{\setcounter{thmlistcnt}{0}%
	\begin{list}{\emph{(\roman{thmlistcnt})}}{%
		\usecounter{thmlistcnt}%
		\setlength{\topsep}{0pt}%
		\setlength{\leftmargin}{0pt}%
		\setlength{\itemsep}{0pt}%
		\setlength{\labelwidth}{17pt}
		\setlength{\itemindent}{30pt}}%
	}%
	{\end{list}}%

\newcommand{\sHH}{\scriptstyle \mathrm{H}}
\newcommand{\sT}{\scriptstyle \mathrm{T}}

\newcommand{\HH}{\mathrm{H}}
\newcommand{\T}{\mathrm{T}}


\begin{document}
\title[A lower bound for the partition function]{A lower bound for the partition function from Chebyshev's inequality applied to a coin flipping model
for the random partition}
\author{Mark Wildon}
\date{\today}

\begin{abstract}
We use a coin flipping model for the random partition and Chebyshev's inequality to prove the
lower bound $\lim \frac{\log p(n)}{\sqrt{n}} \ge C$ for the number of partitions $p(n)$ of $n$,
where $C$ is an explicit constant.
\end{abstract}

\maketitle
\thispagestyle{empty}

A \emph{partition} of a non-negative integer $n \in \N_0$ is a decreasing sequence of natural numbers whose sum is $n$.
Let $p(n)$ be the number of partitions of~$n$. For example, $p(5) = 7$ counts the partitions 
$(5)$, $(4,1)$, $(3,2)$, $(3,1,1)$,
$(2,2,1)$, $(2,1,1,1)$ and $(1,1,1,1,1)$.
In this note we use a model for the random partition
to prove that for all $\epsilon > 0$ there exists $N \in \N$ such that
\[\label{eq:main} \frac{\log p(n)}{\sqrt{n}} >  \frac{\sqrt{8}\log 2}{1 + \epsilon} \quad\text{for all $n \ge N$}.
\tag{$\star$} \]
We end with an explicit  bound that replaces $\sqrt{8} \log 2$ with the slightly smaller constant
$\mfrac{8}{3} \log 2$.  The proof of~($\star$) is self-contained and intended to be readable by anyone knowing the basics
of probability theory.

The asymptotically correct result
is \smash{$\lim_{n \rightarrow \infty} \frac{ \log p(n)}{\sqrt{n}} = 2 \sqrt{\pi^2/6}$}. 
The upper bound
$\log p(n) \le 2 \sqrt{\pi^2/6} \sqrt{n}$ is relatively easy to prove---see for instance Theorem 15.7 in 
\cite{vanLintWilson}---but getting a tight lower bound is much more challenging. A fairly lengthy proof 
using only real analysis was given by Erd{\H{o}}s in~\cite{ErdosPartitions}. Our proof is
motivated by the  model for the random partition 
in~\cite[\S 4.3]{ArratiaTavare}, and by 
the abacus notation for partitions (see \cite[page 79]{JK}). The latter was used in \cite{MarotiLower} to prove
the uniform lower bound $p(n) \ge \mathrm{e}^{2\sqrt{n}}/14$, and in~\cite{WildonAbacus} to prove
the upper bound $\log p(n) \le C(\epsilon) n^{\sfrac{1}{2} + \epsilon}$ for all $\epsilon > 0$.
The novel feature here is to combine these motivations to give a simple proof of~($\star$).

The proof begins with a coin flipping model for the random partition. Using linearity of expectation
it is easy to show
that a partition generated by $m$ flips has expected size about $m^2/8$. Critically, the standard
is of order~$m^{3/2}$. By Chebyshev's inequality, most of the $2^m$ partitions
generated by $m$ coin flips have size within a few standard deviations of $m^2/8$. This leads
easily to the claimed bound.

\subsection*{Coin flipping model}

We represent a partition $\lambda$ of length $\ell$
as the set of \emph{boxes} $\{(i,j) : 1 \le i \le \ell, 1 \le j \le \lambda_i \}$,
forming its \emph{Young diagram}. We draw Young diagrams in  `French notation', so that the box $(i,j)$ 
is
geometrically  a unit square with diagonal from $(i-1,j-1)$ to $(i,j)$.
For example, the partition $(6,4,2,2)$ of length $4$
is shown in Figure~1 above.

%
%
%

\begin{figure}
\begin{center}
\begin{tikzpicture}[x=0.7cm,y=0.7cm]
\draw[thick] (0,0)--(6,0); \draw[thick] (0,1)--(6,1); \draw[thick] (0,2)--(4,2); \draw[thick] (0,3)--(2,3); \draw[thick] (0,4)--(2,4); 
\draw[thick] (0,0)--(0,4); \draw[thick] (1,0)--(1,4); \draw[thick] (2,0)--(2,4); \draw[thick] (3,0)--(3,2); \draw[thick] (4,0)--(4,2); \draw[thick] (5,0)--(5,1); 
\draw[thick] (6,0)--(6,1);
\draw[->] (0,0)--(7.5,0);
\draw[->] (0,0)--(0,4.5);
\node[below left] at (0,0) {$0$}; 
\node[below] at (1,0) {$1$}; \node[below] at (2,0) {$2$}; \node[below] at (3,0) {$3$}; 
\node[below] at (4,0) {$4$}; \node[below] at (5,0) {$5$}; \node[below] at (6,0) {$6$}; \node[below] at (7,0) {$7$}; 
\node[left] at (0,1) {$1$}; \node[left] at (0,2) {$2$}; \node[left] at (0,3) {$3$}; 
\node[left] at (0,4) {$4$}; 
\draw (7,0)--(7,0.2);
\node[above] at (1.5,4) {$\sHH$};
\node[above] at (2.65,2) {$\sHH$};
\node[above] at (3.5,2) {$\sHH$};
\node[above] at (4.65,1) {$\sHH$};
\node[above] at (5.5,1) {$\sHH$};
\node[above] at (6.65,0) {$\sHH$};
\node[right] at (2,3.5) {$\sT$};
\node[right] at (2,2.65) {$\sT$};
\node[right] at (4,1.65) {$\sT$};
\node[right] at (6,0.65) {$\sT$};

\end{tikzpicture}
\end{center}

\vspace*{-8pt}
\caption{The Young diagram of the partition $(6,4,2,2)$ with the corresponding coin flip sequence $\HH\T\T\HH\HH\T\HH\HH\T\HH$.}
\end{figure}

Let $\Omega = \{ \HH, \T \}^m$ be the probability space for $m$ flips of an unbiased coin
in which each $\omega \in \Omega$ has equal probability $\frac{1}{2^m}$.
Given $\omega \in \Omega$ with exactly~$\ell$ tails, 
we define the boundary of a corresponding partition
$p(\omega)$ of length~$\ell$ 
as follows. Start at $(0, \ell)$ and step right to $(1,\ell)$. Then for each head, step one unit right,
and for each tail,
step one unit down. For instance if $m = 10$ and $\omega = \HH\T\T \HH\HH \T \HH\HH \T \HH$ 
then $p(\omega) = (6,4,2,2)$;
 the final head corresponds to a step from $(6,0)$ to $(7,0)$ that is not part of a geometric~box.

Let $N$ be the size of $p(\omega)$ and let $X_t$ be the number of heads up to and including flip $t$.
Let $Y = m - X_m$ be the total number of tails; this is the length of $p(\omega)$. 
A move down at step $t$ adds $X_{t-1} + 1$ boxes to the Young diagram. Therefore setting
\[ C_t = \begin{cases} X_{t-1} & \text{if $\omega_t = \T$} \\
0 & \text{if $\omega_t = \HH$.}\end{cases} \]
we have $N = Y + \sum_{t=1}^m C_t$. 

\subsection*{Expectation and variance}
Since $X_t$ is distributed binomially as $\mathrm{Bin}(t,\mfrac{1}{2})$, we have $\E[X_t] = t/2$ and $\Var X_t = t/4$. Hence
$\E[Y] = m/2$ and $\Var Y = m/4$.
Observe that $C_t = 0$ unless flip $t$ is tails. Conditioning on this event shows that
$\E[C_t] = \frac{1}{2}\E[X_{t-1}]
= \frac{t-1}{4}$. Hence, by linearity of expectation, $\E[N] = \E[Y] + \sum_{t=1}^m \E[C_t] 
= m/2 + \frac{1}{4}
\sum_{t=1}^m (t-1) = m/2 + m(m-1)/8 = m(m+3)/8$.

\begin{lemma}\label{lemma:CXuncorr}
If $t \le u$ then the random variables $C_t$ and $X_u$ are uncorrelated.
\end{lemma}

\begin{proof}
Again we  condition on the event that flip $t$ is tails. In this 
event, $C_t = X_{t-1}$ and $X_u = X_{t-1} + W$, where $W$ is the number
of heads between flips $t+1$ and~$u$, inclusive. Since $W$ is independent of $X_{t-1}$,
\allowdisplaybreaks
\begin{align*} 
\E[C_t X_u] &= \frac{1}{2} \E[X_{t-1}(X_{t-1} + W)] 
\\ &= \frac{1}{2} \bigl( \E[X_{t-1}^2] + \E[X_{t-1}] \E[W]\bigr) 
\\ &= \frac{1}{2} \bigl( \Var X_{t-1} + \E[X_{t-1}]^2 + \E[X_{t-1}]\E[W] \bigr) 
\\ &= \frac{1}{2} \left( \frac{t-1}{4} + \Bigl( \frac{t-1}{2} \Bigr)^2 + \Bigl(\frac{t-1}{2} \Bigr) \Bigl(\frac{u-t}{2}\Bigr) \right) 
\\ &= \frac{1}{2} \Bigl( \frac{t-1}{4} \Bigr) \bigl( 1 + t-1 + u-t \bigr) 
\\ &= \E[C_{t-1}] \E[X_u] \end{align*}
as required.
\end{proof}

As a corollary, using a similar conditioning argument, we find that
\[ \E[C_t C_u] = \frac{1}{2} \E[X_{t-1} C_u ] = \frac{1}{2} \E[X_{t-1}] \E[C_u] = \E[C_t][C_u] \]
whenever $t < u$. Hence $C_t$ and $C_u$ are uncorrelated for distinct $t$ and $u$.
This is perhaps a little surprising, since the inequality $C_t \le C_u$ for $t < u$ shows
that they are  not independent, in general.
A final conditioning argument shows that
$\Var C_t = \E[C_t^2] - \E[C_t]^2 = \frac{1}{2} \E[X_{t-1}^2] - \frac{1}{4} \E[X_{t-1}]^2$,
and so
\[ \begin{split} \Var C_t = \frac{1}{2} \Var X_{t-1} + \frac{1}{4} \E[X_{t-1}]^2 &= 
\frac{1}{2} \Bigl( \frac{t-1}{4} \Bigr) + \frac{1}{4} \Bigl( \frac{t-1}{2} \Bigr)^2 
\\ &\qquad = 
\Bigl( \frac{t-1}{16} \Bigr) \bigl( 2 + t-1 \bigr) = \frac{t^2-1}{16}. \end{split} \]
By Lemma~\ref{lemma:CXuncorr},
$\E[C_t Y] = \E[C_t (m-X_m)] = \E[C_t] \E[m-X_m] = \E[C_t] \E[Y]$
for all $t$. Hence $C_t$ and $Y$ are also uncorrelated.
If
$Z$ and $Z'$ are uncorrelated random variables then, by a one-line calculation, $\Var (Z+Z') = \Var Z + \Var Z'$.
We therefore have
$\Var N = \Var Y + \sum_{t=1}^m \Var C_t$ 
and so
\[ \Var N = \frac{m}{4} + \sum_{t=1}^m \frac{t^2-1}{16} = \frac{3m}{16} + \frac{m(m+1)(2m+1)}{96} =
\frac{m^3}{48} + \frac{m^2}{32} + \frac{19m}{96}. \]
Critically $\Var N$ is cubic in $m$, not quartic as one might naively expect.
To simplify calculations, we use the crude upper bound $m^3/48 + m^2/32 + 19m/96 \le 2m^3/96 + 4m^3 / 96 = m^3/16$
for $m \ge 3$ to get $\Var N \le m^3 /16$.

\subsection*{Lower bound}
The concentration of measure estimate in Chebyshev's inequality
\[ \P\Bigl[ \, \bigl| Z- \E[Z] \bigr| \ge d \sqrt{\Var Z}\, \Bigr] \le \frac{1}{d^2} \]
implies that
\[ \P\Bigl[ \bigl| N - \frac{m(m+3)}{8} \bigr| \ge d \frac{m^{3/2}}{4} \Bigr] \le \frac{1}{d^2} \]
for $m \ge 3$ and any $d > 0$.

The probability space $\Omega$ has $2^m$ elements. The proportion giving partitions
with $\bigl| N - m(m+3)/8 \bigr|
< d m^{3/2} /4$ is more than $1-1/d^2$. Since distinct coin flip sequences
give distinct partitions, it follows that
\[ \sum_{n} p(n) > 2^m \bigl( 1 - \frac{1}{d^2} \bigr) \]
where the sum is over all $n \in \N_0$ such that 
$| n - m(m+3)/8 | < d m^{3/2} /4$. Since $p(n)$ is monotonic, we deduce that, for $m \ge 3$,
\[ p\Bigl( \frac{m(m+3)}{8}   + d \frac{m^{3/2}}{4} \Bigr) > \frac{2^m}{d m^{3/2}} 2\bigl( 1 - \frac{1}{d^2} \bigr), \]
where we extend the domain of $p$ to $\R$ by setting $p(x) = p(\lfloor x \rfloor)$.
The function $d\mapsto \frac{1}{d}\bigl(1 -\frac{1}{d^2} \bigr)$ is maximized when $d = \sqrt{3}$,
where it has value \smash{$\mfrac{2}{3\sqrt{3}}$}. Therefore  we take $d = \sqrt{3}$. 
Let $\eta > 0$ be given.
Provided $m$ is sufficiently
large we have $3m/8 + \sqrt{3}\hskip1pt m^{3/2} / 4 < \eta m^2/8$.
Hence
\[ \tag{$\dagger$} p\Bigl( \frac{m^2}{8} (1+\eta) \Bigr) > 2^m \frac{4}{3\sqrt{3} m^{3/2}}  \]
for all $m$ sufficiently large. Setting $n = m^2(1+\eta)/8$ and taking logs we obtain
\[ \log p(n) \ge \sqrt{\frac{8n}{1+\eta}} \log 2 - \frac{3}{2}  \log \frac{8n}{1+\eta}
+ \log \frac{4}{3\sqrt{3}} \]
for all $n$ sufficiently large.
Since
$(\log n) / \sqrt{n} \rightarrow 0$ as $n \rightarrow \infty$ 
it follows that for all $\epsilon > 0$, 
\[ \frac{ \log p(n)}{\sqrt{n}} > \frac{\sqrt{8} \log 2}{1 + \epsilon} \]
for all $n$ sufficiently large, as claimed in~\eqref{eq:main}. The constant on the right-hand side is approximately $1.961$, somewhat 
lower
than the asymptotically correct $2 \sqrt{\pi^2/6} \approx 2.565$.
For a concrete lower bound, take
 $\eta = \frac{1}{8}$ and
$m = 8\sqrt{n}/3$ in ($\dagger$) to get
$p(n) \ge 2^{8\sqrt{n}/3} / 2^{5/2} n^{3/4}$ 
for all $n$ sufficiently large. (One can easily check that $ n\ge 10^6$ suffices.)
Using a computer to check small cases one can show that in fact
\[ p(n) \ge \frac{2^{8\sqrt{n}/3}}{2^{5/2} n^{3/4}} \quad\text{for all $n \ge 2$.} \] 


%
%
%


\enlargethispage{9pt}
\smallskip
\providecommand{\href}[2]{#2}
\renewcommand{\MR}[1]{\relax{} }


\begin{thebibliography}{1}

\bibitem{ArratiaTavare}
Richard Arratia and Simon Tavar{\'e}, \emph{The cycle structure of random
  permutations}, Ann. Probab. \textbf{20} (1992), no.~3, 1567--1591.
  \MR{1175278 (93g:60013)}

\bibitem{ErdosPartitions}
P.~Erd{\H o}s, \emph{On an elementary proof of some asymptotic formulas in the
  theory of partitions}, Ann. of Math. (2) \textbf{43} (1942), 437--450.
  \MR{MR0006749 (4,36a)}

\bibitem{JK}
G.~James and A.~Kerber, \emph{The representation theory of the symmetric
  group}, Encyclopedia of Mathematics and its Applications, vol.~16,
  Addison-Wesley Publishing Co., Reading, Mass., 1981. 

\bibitem{MarotiLower}
A.~Mar{\'o}ti, \emph{On elementary lower bounds for the partition function},
  Integers \textbf{3} (2003), A10, 9 pp. (electronic). \MR{MR2006609
  (2004g:05014)}
  
\bibitem{vanLintWilson}
J.~H. van Lint and R.~M. Wilson.
{\em A Course in Combinatorics}, 2nd edition,
Cambridge University Press, 2001.

\bibitem{WildonAbacus}
Mark Wildon, \emph{Counting partitions on the abacus}, Ramanujan J. \textbf{17}
  (2008), no.~3, 355--367. \MR{2456839 (2009k:05021)}

\end{thebibliography}


\end{document}